\documentclass[a4paper,12pt]{amsart}
\usepackage{cite}
\usepackage{amsmath}
\usepackage{amsthm}
\usepackage{amsfonts}
\usepackage{mathrsfs}
\usepackage{geometry}
\usepackage{amssymb}
\usepackage{amsfonts}
\usepackage{tikz-cd}
\usepackage{romannum}
\usepackage{titlesec}
\usepackage{enumitem}
\usepackage{xcolor}
\usepackage{bm}
\newtheorem{theorem}{Theorem}[section]
\newtheorem{corollary}[theorem]{Corollary}
\newtheorem{lemma}[theorem]{Lemma}
\newtheorem{definition}[theorem]{Definition}
\newtheorem{proposition}[theorem]{Proposition}
\newtheorem{remark}[theorem]{Remark}

\titleformat{\section}       
{\normalfont\centering\large\bfseries} 
{\thesection}              
{0.5em}                      
{}          
{}

\titleformat{\subsection}
{\bfseries}
{\thesubsection}
{1em}
{}

\author{Liqingjing Wang}

\address{Liqingjing Wang, Universit\'e C\^ote d'Azur,  LJAD, France}
\email{liqingjing.wang@univ-cotedazur.fr}

\title{Deformations of canonical bundle for smooth weakly K\"ahler morphisms} 
\date{\today}

\begin{document}
\begin{abstract} 
Let $\pi: \mathcal X \rightarrow \Delta$ be a smooth proper family of compact complex manifolds
such that the central fiber $\mathcal{X}_0$ is K\"ahler. Then all the fibers close to $0$ are K\"ahler and $\pi$ is a weakly K\"ahler morphism, even if $\mathcal X$ is not a K\"ahler space.
In this paper we show that if $\dim \mathcal{X}_0=4$ and $K_{\mathcal{X}_0}$ is nef, then $K_{\mathcal X_t}$ is nef for all $t$ in a neighbourhood of the origin.
\end{abstract}
\maketitle
	\pagenumbering{arabic}
	\section{Introduction}

\noindent Let $\Delta$ be a unit disc in $\mathbb{C}$ and let $\pi:\mathcal{X}\longrightarrow\Delta$ be a proper holomorphic submersion. We say that $\mathcal{X}\longrightarrow\Delta$  is a family of deformations of the fiber $\mathcal{X}_0$ and each fiber $\mathcal{X}_t$, $t\in\Delta$, is called a deformation of $\mathcal{X}_0$. In addiction, if $\pi:\mathcal{X}\longrightarrow\Delta$ is a projective morphism, then $\pi:\mathcal{X}\longrightarrow\Delta$ is called a projective family.
In this paper we consider the following question:

\medskip

\noindent \textbf{Question:} Let $\pi:\mathcal{X}\longrightarrow\Delta$ be a deformation family of a compact K{\"a}hler manifold $X$ of dimension $n$. If the canonical line bundle $K_X$ is nef, is $K_{\mathcal{X}_t}$ nef for all $t$ in a small neighborhood of $0$?

\medskip

If $\pi:\mathcal{X}\longrightarrow\Delta$ is a projective morphism, J.A Wi{\'s}niewski \cite[Thm.3]{zbMATH05636215} proved that if the canonical line bundle of one fiber is not nef, then none of canonical line bundles of fibers in $\pi:\mathcal{X}\longrightarrow\Delta$ are nef.

Recently, \cite[Thm.3.9]{arXiv:2510.23967} gave a positive answer for the smooth projective family $\pi:\mathcal{X}\longrightarrow\Delta$ for arbitrary dimensions.  \cite[Lem.1.1]{zbMATH00058549} plays a crucial role in their work. 
 However,   \cite[Lem.1.1]{zbMATH00058549} is hard to generalize into the K{\"a}hler setting  because a real $(1,1)$-class $\omega_t$ changes continuously in $H^2(\mathcal{X}_0,\mathbb{R})$. Therefore, to investigate the weakly K{\"a}hler case, we might need to try another approach.  One possible approach is how $K_{\mathcal{X}_t}$-negative rational curves deform with a deformation family $\pi:\mathcal{X}\longrightarrow\Delta$.  In \cite{arXiv:2510.23967}, they also used  MMP for three dimensional compact K{\"a}hler manifolds to achieve this goal and gave a positive answer for a family of smooth K{\"a}hler manifolds of dimension small or equal to three.
We give a positive answer for $n=4$:

\begin{theorem}\label{nefness preserves by deformation THM}
	Let $\mathcal{X}_{0}$ be a four dimensional compact K{\"a}hler manifold with $K_{\mathcal{X}_0}$ nef.
	Let $\pi:\mathcal{X}\longrightarrow\Delta$ be a deformation of $\mathcal{X}_0$ over the unit disc $\Delta\subset\mathbb{C}$. 
	Then there exists a small neighborhood $\Delta'$ of $0$, such that $K_{\mathcal{X}_t}$ is nef for all $t\in\Delta'$.
\end{theorem}

Note that our result doesn't require any K{\"a}hler property for the total space $\mathcal{X}$. Notice that a total space $\mathcal{X}$ isn't a K{\"a}hler manifold in general, a counterexample is showed in  \cite[Example.3.9]{zbMATH03935471}. However, the manifold $\mathcal X_t$ is compact K\"ahler for $t$ close to zero, which is a classical result in \cite{zbMATH05016783}, and the morphism is always weakly K\"ahler (cf. \cite[Thm.6.3]{zbMATH03935470}).

Our work is based on a contraction theorem for Compact K{\"a}hler fourfolds in order to give a detailed picture of the deformation of $K_{\mathcal{X}_t}$-negative rational curves with $\mathcal{X}_t$.  With a clear picture of the exceptional locus, we obtain the main technical result of this paper which implies Theorem \ref{nefness preserves by deformation THM}:

\begin{theorem}\label{general fibre of contraction deforms with the deformation family THM}
	Let $\mathcal{X}_{0}$ be a four dimensional compact K{\"a}hler manifold with $K_{\mathcal{X}_0}$ nef.
	Let $\pi:\mathcal{X}\longrightarrow\Delta$ be a deformation of $\mathcal{X}_0$ over the unit disc $\Delta\subset\mathbb{C}$. Let $t\in\Delta$ be an arbitrary point contained in a small neighborhood of $0$. If $K_{\mathcal{X}_t}$ is pseudo-effective and not nef, then there is a subvariety $F\subset \mathcal{X}_t$ with $-K_{\mathcal{X}_t}|_F$ ample such that $F$ deforms with $\mathcal{X}_t$.
\end{theorem}

When this paper was almost finished, the paper \cite{li2026deformationinvariancecanonicalnefness} was posted on the arXiv which gives a positive answer for K\"ahler morphisms using different arguments. Note that in general the deformation
$\pi$ is not a K\"ahler morphism.

\subsection*{Acknowledgements:} 
\noindent The author wishes to express her sincere gratitude to her advisor Andreas H\"oring for his guidance, encouragement and fruitful discussions.
	\section{Preliminaries}

\noindent We first introduce some notations which are used in the cone theorem and contraction theorem in the K{\"a}hler setting 
\begin{definition}\label{Bott chern cohomology DEF}
	Let $X$ be an irreducible and reduced complex space. Let $\mathcal{H}_X$ be the sheaf of real parts of holomorphic functions multiplied with  $i$. A $(1,1)$-form(resp. $(1,1)$-current) with local potentials on $X$ is a global section of the quotient sheaf $\mathcal{A}_{X}^0/\mathcal{H}_X$(resp. $\mathcal{D}_X/\mathcal{H}_X$), we define the Bott-Chern cohomology 
	\[
	H^{1,1}_{\text{BC}}(X):=H^1(X,\mathcal{H}_X).
	\]
\end{definition}

\begin{definition}\label{N_1(X) and overline{NA}(X) DEF}
	Let $X$ be a normal compact K{\"a}hler space. We define $N_1(X)$ to be the vector space of real closed currents of bidimenison $(1,1)$ modulo the following equivalent relation: $T_1\equiv T_2$ if and only if
	\[
	T_1(\eta)=T_2(\eta)
	\]
	for all real closed $(1,1)$-forms $\eta$ with local potentials.
	We define $\overline{NA}(X)\subset N_1(X)$ to be the closed cone generated by the classes of possitive currents.
\end{definition}

\begin{definition}\cite[Chap.\Romannum{2}.Def.2.11]{zbMATH00833161}\label{Deformations of rational curves DEF}
	Let $X$ be a compact K{\"a}hler manifold. Let $\text{Hom}_{bir}(\mathbb{P}^1,X)$ be the space of birational morphisms from $\mathbb{P}^1$ to $X$. Let 
	\[
	\text{Hom}_{bir}(\mathbb{P}^1,X)=\cup_{i}W_i
	\]
	be the decomposition into irreducible subschemes. Let $\overline{V}_i\subset\text{Chow}(X)$ be the image closure of the image of $W_i$. Let $V_i\subset \overline{V}_i$ be the open subset parametrizing irreducible $1$-cycles. We define the space of rational curves on $X$ as
	\[
	\text{RatCurves}^n(X)=\cup_{i}{V_i}^n.
	\]
	The superscript $n$ represent the normalization.
\end{definition}

\noindent We have the following Corollary of \cite[Cor.3.2, Cone Theorem]{arXiv:2404.12007}.
\begin{corollary}\label{cone theorem without pair for Kahler 4-fold COR}
	Let $X$ be a four dimensional compact K{\"a}hler manifold. Assume the canonical line bundle $K_X$ is pseudo-effective. Then there are at most countably many rational curves $\Gamma_j$ such that
	\[
	0<-K_{X}\cdot\Gamma_j\leq 8
	\]
	for all $j\in J$ and 
	\[
	\overline{NA}(X)=\overline{NA}(X)_{{K_X}_{\geq0}}+
	\displaystyle\sum_{j\in J}\mathbb{R}^{+}[\Gamma_j].
	\]
	Moreover, the irreducible component $V_i\subset\text{RatCurves}^n(X)$ which contains $[\Gamma_j]$ has dimension at least two for every $j\in J$.
\end{corollary} 
\begin{proof}
	Since the line bundle $K_X$ is pseudo-effective, we take the boundary $B+\beta_X$ in \cite[Cor.3.2, Cone Theorem]{arXiv:2404.12007} to be trivial and conclude the result.
	
	\bigskip
	
	\noindent Let $j$ be an arbitrary member in $J$. Let $f$ be the birational morphism from $\mathbb{P}^1$ to $\Gamma_j\subset\mathcal{X}_t$. Let $Hom_{\text{bir},[f]}(\mathbb{P}^1,X)$ be the irreducible subvariety in $Hom_{\text{bir}}(\mathbb{P}^1,X)$ which contains $[f]$.
	By the Riemann-Roch formula, we know:
	\begin{align*}
		\text{dim}(Hom_{\text{bir},[f]}(\mathbb{P}^1,X))
		&\geq\chi(\mathbb{P}^1,f^* T_X)\\
		&=\text{det}(f^*T_X)+\text{dim}(X)\\
		&=-K_{X}\cdot\Gamma_j+\text{dim}(X)\\
		&\geq 5
	\end{align*}
	By \cite[Chap.\Romannum{2}.Thm.2.15]{zbMATH00833161}, the morphism
	\[
	Hom_{\text{bir}}(\mathbb{P}^1,X)\longrightarrow \text{RatCurves}^n(X)
	\] 
	is smooth of relative dimension three with connected 
	fibers. Hence, the irreducible component which contains $[\Gamma_j]$ has dimension at least two.
\end{proof}

\begin{lemma}\cite[Prop.8.1.(b)]{zbMATH06541951}\label{relatively numerical trivial line bundle is the pull back of the line bundle on the base LEM}
	Let $X$ be a a compact K{\"a}hler fourfold. Let $\mathbb{R}^{+}[\Gamma_i]$ be a $K_X$-negative extremal ray in $\overline{NA}(X)$. Suppose that there exists a morphism $\varphi:X\longrightarrow Y$ onto a normal complex space such that $-K_X$ is $\varphi$-ample and a curve $C\subset X$ is contracted if and only if $[C]\in\mathbb{R}^{+}[\Gamma_i]$. Then we have the exact sequence:
	\[
	\begin{tikzcd}[column sep=normal]
		0\arrow[r]&\text{Pic}(Y)\arrow[r,"\varphi^*"]&\text{Pic}(X)\arrow[black]{r}{[L]\mapsto L\cdot\Gamma_i}&\mathbb{Z}
	\end{tikzcd}
	\]
\end{lemma}
\begin{theorem}\label{Cohomology vanishing of degree 2 hypersurface THM}
	Let $X\subset\mathbb{P}^n,\;n\geq 3$ be a hypersurface of degree two which is not necessarily normal. Then we have 
	\[
	H^1(X,\mathcal{O}_{\mathbb{P}^n}(m))=0,\;for \, all\, m>-(n-1).
	\]
\end{theorem}
\begin{proof}
	
	We twist the following short exact sequence with $\mathcal{O}_{\mathbb{P}^n}(m)$
	\[
	\begin{tikzcd}[column sep=small]
		0\arrow[r]&\mathcal{O}_{\mathbb{P}^n}(-2)\arrow[r]&\mathcal{O}_{\mathbb{P}^n}\arrow[r]&\mathcal{O}_{X}\arrow[r]&0,
	\end{tikzcd}
	\]
	and get
	\[
	\begin{tikzcd}[sep=scriptsize]
		0\arrow[r]&\mathcal{O}_{\mathbb{P}^n}(m-2)\arrow[r]&\mathcal{O}_{\mathbb{P}^n}(m)\arrow[r]&\mathcal{O}_{X}(m)\arrow[r]&0.
	\end{tikzcd}
	\]
	We apply well-known vanishing results for line bundles on $\mathbb{P}^n$ and the long exact sequence in cohomology.
\end{proof}
	\section{Basic Fact}
\begin{proposition}\label{nef supporting class for extremal ray PRO}
	Let $X$ be a four dimensional compact K{\"a}hler manifold. Assume the canonical line bundle $K_X$ is pseudo-effective but not nef. Let $\{\Gamma_{i}\}_{i\in I}$ be at most countably many rational curves in $X$ such that
	\[
	0<-K_{X}\cdot\Gamma_i\leq 8
	\]
	for all $i\in I$ and 
	\[
	\overline{NA}(X)=\overline{NA}(X)_{{K_X}_{\geq0}}+
	\displaystyle\sum_{i\in I}\mathbb{R}^{+}[\Gamma_i].
	\]
	Let $\Gamma_{i_0}$ be an  arbitrary member in $\{\Gamma_{i}\}_{i\in I}$. Then there exists a nef class $\alpha\in N^1(X)$ such that 
	\[
	\mathbb{R}^+[\Gamma_{i_0}]=\{z\in\overline{NA}(X)\text{ }|\text{ }\alpha\cdot z=0\},
	\]
	and the class $\alpha$ is strictly positive on
	\[
	\biggl(\overline{NA}(X)_{{K_{X}}_{\geq0}}+
	\displaystyle\sum_{i\in I,i\neq i_0}\mathbb{R}^{+}[\Gamma_i]\biggl)\setminus\{0\}
	\]
	We call $\alpha$ a nef supporting class for the extremal ray $\mathbb{R}^+[\Gamma_{i_0}]$.
\end{proposition}
\begin{proof}
	We set 
	\[
	V:= \overline{NA}(X)_{{K_{X}}_{\geq0}}+
	\displaystyle\sum_{i\in I,i\neq i_0}\mathbb{R}^{+}[\Gamma_{i}],
	\]
	We have
	\[
	\overline{NA}(X)=V+\mathbb{R}^{+}[\Gamma_{i_0}]
	\]
	Since we have
	\[
	0<-K_{X}\cdot\Gamma_{i}\leq8,
	\] 
	by \cite[Lem.6.1]{zbMATH06541951},
	the cone $V$ is closed. 
	Therefore, there exists a linear form on $N_1(X)$ which vanishes on $R$ and is positive on $V\setminus \{0\}$ by \cite[Lem. 6.7(d)]{zbMATH01634463}. This linear form gives the class $\alpha$ by \cite[Prop.3.9]{zbMATH06541951}.
\end{proof}

\begin{definition}\label{Small and divisorial ray DEF}
	Let $X$ be a four dimensional compact K{\"a}hler manifold. Let $\mathbb{R}^+[\Gamma_{i}]$ be an extremal ray in $\overline{NA}(X)$. We set $\cup_{C\in \mathbb{R}^+[\Gamma_{i}]}C$ to be a union of all curves which are contained in the extremal ray $\mathbb{R}^+[\Gamma_{i}]$.  We define $\mathbb{R}^+[\Gamma_{i}]$ is small if every connected component of $\cup_{C\in \mathbb{R}^+[\Gamma_{i}]}C$ has codimension at least two and it is divisorial if  there exists a connected component of $\cup_{C\in \mathbb{R}^+[\Gamma_{i}]}C$ is codimensional one.
\end{definition}

\begin{definition}\label{Generical contraction exists DEF}
	Let $X$ be a four dimensional compact K{\"a}hler manifold with $K_X$ pseudo-effective but not nef. We fix a $K_X$-negative extremal ray $\mathbb{R}^+[\Gamma_{i_0}]$. Let $\alpha$ be the nef supporting class of $\mathbb{R}^+[\Gamma_{i_0}]$.  We say that a contraction for $\mathbb{R}^+[\Gamma_{i_0}]$ exists  when the following properties hold:
	\begin{itemize}
		\item[1.] If the extremal ray $\mathbb{R}^+[\Gamma_{i_0}]$ is divisorial, the null locus $\text{Null}(\alpha)$ is a prime Cartier divisor $D$.  Moreover, the following holds:
		\begin{itemize}
			\item[1.1] There is a proper holomorphic fibration $\varphi: D\rightarrow B$ from $D$ to a normal analytic space $B$. The divisor $-D$ is $\varphi$-ample.
			\item[1.2] 	 There is a bimeromorphic morphism $f$ from $X$ to a normal complex space $Y$ such that $f|_{X\setminus D}$ is an isomorphism and $f|_{D}=\varphi$. 
		\end{itemize} 
		\item[2.] If the extremal ray $\mathbb{R}^+[\Gamma_{i_0}]$ is small, there exists a bimeromorphic morphism from $X$ to a normal complex space $Y$ such that $f(\text{Null}(\alpha))$ is a finite set of points and $f|_{X\setminus \text{Null}(\alpha)}$ is an isomorphism.
		\item[3.] The bimeromorphic map $f$ on $X$ satisfies:
		\begin{itemize}
			\item[3.1.] $f$ has connected fibres;
			\item[3.2.] the anti canonical line bundle $-K_X$ is $f$-ample;
			\item[3.3.] a curve $C\subset X$ is contracted by $f$ if and only if $[C]\subset\mathbb{R}^+[\Gamma_{i_0}]$.
		\end{itemize}
	\end{itemize}
	
\end{definition}

\begin{definition}\label{Type introduction of generic contraction DEF}
	Let $X$ be a four dimensional compact K{\"a}hler manifold with $K_X$ pseudo-effective but not nef. We fix a $K_X$-negative extremal ray $\mathbb{R}^+[\Gamma_{i_0}]$. Let $\alpha$ be the nef supporting class of $\mathbb{R}^+[\Gamma_{i_0}]$.  Assume there exists a divisorial contraction $f:X\longrightarrow Y$ for $\mathbb{R}^+[\Gamma_{i_0}]$. We say $f$ is of the type $(3,n)$ if the base $B$ is n-dimensional.
\end{definition}

The following result is a special case of \cite[Thm.1.4]{hacon2026kahlermmptranscendentalbasepointfree}. In my forthcoming thesis, I will present a self-contained proof based on the analysis of the geometry of the null locus of the nef supporting class:

\begin{theorem}\label{Generic contraction exists for 4-fold THM}
	Let $X$ be a four dimensional compact K{\"a}hler manifold. Assume that the canonical line bundle $K_{X}$ is pseudo-effective and not nef. Let $\mathbb{R}^+[\Gamma_{i_0}]$ be a $K_{X}$-negative extremal ray. Then the contraction of $\mathbb{R}^+[\Gamma_{i_0}]$ exists.
\end{theorem}

\begin{proposition}\label{Pseudo-effective preserves in the deformation PRO}
	Let $\mathcal{X}_0$ be a four dimensional compact Kähler manifold with
	$K_{\mathcal{X}_0}$ nef. 
	Let $\pi:\mathcal{X}\longrightarrow\Delta$ be a deformation of $\mathcal{X}_0$ over a unit disc $\Delta\subset\mathbb{C}$. Let $\mathcal{X}_t$ be an arbitrary fibre of the deformation $\pi:\mathcal{X}\longrightarrow\Delta$. Then the canonical line bundle $K_{\mathcal{X}_t}$  is pseudo-effective.
\end{proposition}
\begin{proof}
	We argue by contradiction. Assume there is a fibre $\mathcal{X}_s$ such that $K_{\mathcal{X}_s}$ is not pseudo-effective. By \cite[Thm.1.1]{arXiv:2501.18088}, the manifold $\mathcal{X}_s$ is uniruled. Then every fibre of $\pi:\mathcal{X}\longrightarrow\Delta$ is uniruled by \cite[Chap.\Romannum{4}.Cor.1.10]{zbMATH00833161}. Hence, for the center fibre $\mathcal{X}_0$, the canonical divisor $K_{\mathcal{X}_0}$ is not pseudo-effective. Thus it is not nef, a contradiction.
\end{proof}

\begin{proposition}\label{countable many fibres with canonical divisor are not nef PRO}
	Let $\mathcal{X}_0$ be a four dimensional compact Kähler manifold with
	$K_{\mathcal{X}_0}$ nef. 
	Let $\pi:\mathcal{X}\longrightarrow\Delta$ be a deformation of $\mathcal{X}_0$ over a unit disc $\Delta\subset\mathbb{C}$.  There exists a small neighborhood $\Delta'$ of $0$ and a set $Z\subset\Delta'$ containing at most countably many points such that $K_{\mathcal{X}_t}$ is nef for any $t\in\Delta'\setminus Z$.
\end{proposition}
\begin{proof}
	By \cite[Prop.2.6]{zbMATH06216389}, there is a small neighborhood $\Delta'\subset \Delta$ of $0$ such that every connected component of the relatively Chow space $\text{Chow}(\pi^{-1}(\Delta')/\Delta')$ is proper over $\Delta'$. Moreover, by \cite[Chap.\Romannum{8}.Thm.1.10]{zbMATH00611963}, there are only countably many irreducible components contained in $\text{Chow}(\pi^{-1}(\Delta')/\Delta')$. We set the decomposition into irreducible components to be
	\[
	\text{Chow}(\pi^{-1}(\Delta')/\Delta'):=\cup_{i\in I} \mathcal{D}_i.
	\] 
	
	\noindent Let $\mathcal{U}_i\subset\mathcal{D}_i\times \pi^{-1}(\Delta')$ be the universal family of $\mathcal{D}_i$. Consider the commutative diagram 
	\[
	\begin{tikzcd}
		\mathcal{U}_i\arrow[r,"pr_2"]\arrow[d,"\mu",swap]&\pi^{-1}(\Delta')\arrow[d,"\pi"]\\
		\mathcal{D}_i\arrow[r,"\pi_*"]&\Delta'
	\end{tikzcd}      
	\]
	Since the morphisms $\pi_*$, $\mu$ and $\pi$ are proper, the morphism $pr_2$ is proper. By the proper mapping theorem of Remmert, the image $\pi\circ pr_2(\mathcal{U}_i)$ is an analytic set in $\Delta'$, thus it is either a point in $\Delta'$ or $\Delta'$ itself. Moreover, if the image $\pi\circ pr_2(\mathcal{U}_i)$ is $\Delta'$, then for every fibre $\mathcal{X}_t$, there is a cycle parameterized by $\mathcal{D}_i$ that is also contained in the fibre $\mathcal{X}_t$.
	We set
	\[
	Z:=\{p\in\Delta'\,|\,\exists\, i\in I\text{ such that }\pi\circ pr_2(\mathcal{U}_i)=p\}
	\]
	The set $Z$ contains at most countably many points because there are only countably many irreducible components contained in $\text{Chow}(\pi^{-1}(\Delta')/\Delta')$.
	
	\bigskip
	
	\noindent Now we argue by contradiction.
	Assume there is a fibre $\mathcal{X}_{t_0}$ with $K_{\mathcal{X}_{t_0}}$ not nef and $\pi(\mathcal{X}_{t_0})\in\Delta'\setminus Z$. Then by Theorem \ref{Pseudo-effective preserves in the deformation PRO} , we know $K_{\mathcal{X}_{t_0}}$ is pseudo-effective and not nef. By \cite[Thm.1.3]{zbMATH07147342}, there exists a rational curve $C_{t_0}$ in $\mathcal{X}_{t_0}$ such that
	\[
	K_{\mathcal{X}_{t_0}}\cdot C_{t_0}<0.
	\]
	Let $\mathcal{D}_{i_0}$ be the connected component in $\text{Chow}(\pi^{-1}(\Delta')/\Delta')$ that contains $[C_{t_0}]$ and let $\mathcal{U}_{i_0}$ be the universal family of $\mathcal{D}_{i_0}$. Since we have  $\pi(\mathcal{X}_{t_0})\in\Delta'\setminus Z$, by the definition of $Z$,  the morphism $\pi\circ pr_2|_{\mathcal{U}_{i_0}}:\mathcal{U}_{i_0}\rightarrow \Delta$ is surjective.
	Therefore, for every $t\in\Delta'$, there exists 1-cycle $[C_t]\in\mathcal{U}_{i_0}$ contained in the fibre $\mathcal{X}_t$ and thus we have
 	\[
	   [C_t]=[C_{t_0}]\in H_2(\mathcal{X},\mathbb{Z}).
	\]
	Since the cohomology class $c_1(K_{\mathcal{X}_t})$ is constant in $H^2(\mathcal{X},\mathbb{Z})$, there is an integer $a$ such that for every $t\in\Delta'$, we have
	\[
	K_{\mathcal{X}_t}\cdot [C_t]\equiv a.
	\]
	In particular, for the central fibre $\mathcal{X}_0\simeq X$, there is an 1-cycle $[C_0]\in H_2(X,\mathbb{Z})$ such that
	\[
	K_{\mathcal{X}_0}\cdot [C_0]<0.
	\] 
	Thus $K_{\mathcal{X}_0}$ is not nef, a contradiction.
\end{proof}

\bigskip
	\section{Proof of Theorem \ref{nefness preserves by deformation THM} and Theorem \ref{general fibre of contraction deforms with the deformation family THM} }

\noindent In this section, we show how to deduce Theorem \ref{nefness preserves by deformation THM} and Theorem  \ref{general fibre of contraction deforms with the deformation family THM}  from Theorem \ref{Generic contraction exists for 4-fold THM}

\begin{proof}[Proof of Theorem \ref{general fibre of contraction deforms with the deformation family THM}]
	Let $\Delta'$ be the small neighborhood of $0$ that we select in Proposition \ref{countable many fibres with canonical divisor are not nef PRO}. Let $t$ be an arbitrary point in $\Delta'$. By Proposition \ref{Pseudo-effective preserves in the deformation PRO}, $K_{\mathcal{X}_t}$ is pseudo-effective and not nef. By Theorem \ref{Generic contraction exists for 4-fold THM}, the contraction for a $K_{\mathcal{X}_t}$-negative extremal ray $\mathbb{R}^+[\Gamma_{i_0}]$ exists.
	Let $f:\mathcal{X}_t\longrightarrow Y$ be the contraction. We say that a subvariety $F\subset\mathcal{X}_t$ satisfies \emph{Condition (A)} if the following holds: 
	\begin{itemize}
		\item $F$ is locally complete intersection.
		\item $H^1(F,N_{F/\mathcal{X}_t})=H^1(F,N_{F/\mathcal{X}})=0$
		\item $h^0(F,N_{F/\mathcal{X}_t})=d$ and $h^0(F,N_{F/\mathcal{X}})=d+1$.
	\end{itemize}
	We claim that there is a subvariety in the exceptional locus of the contraction which satisfies \emph{Condition (A)}.
	
	\medskip
	
	\noindent Assume the claim for the time being, we finish the proof. Let $\mathcal{H}_F$ be the Hilbert scheme by which $F$ is parameterized. The dimension of $\mathcal{H}_F$ is $d+1$ because we have
	\[
	h^0(F,N_{F/\mathcal{X}_t})=d\text{ and }h^0(F,N_{F/\mathcal{X}})=d+1.
	\] 
	Since we have 
	\[
	H^1(F,N_{F/\mathcal{X}_t})=H^1(F,N_{F/\mathcal{X}})=0,
	\]
	the morphism from $\mathcal{H}_F$ to $\text{Chow}(\pi^{-1}(\Delta')/\Delta')$ is an isomorphism.
	Let
	\[
	\mathcal{D}_F\subset \text{Chow}(\pi^{-1}(\Delta')/\Delta')
	\] 
	be the image of $\mathcal{H}_F$ in  $\text{Chow}(\pi^{-1}(\Delta')/\Delta')$.  $\mathcal{D}_F$ is an irreducible  ($d+1$)-dimensional variety in $\text{Chow}(\pi^{-1}(\Delta')/\Delta')$. Let $\mathcal{U}_F\subset \mathcal{D}_F\times\pi^{-1}(\Delta')$ be the universal family of $\mathcal{D}_F$. Consider the diagram
	\[
	\begin{tikzcd}
		\mathcal{U}_F\arrow[r,"pr_2"]\arrow[d,"\mu",swap]&\pi^{-1}(\Delta')\arrow[d,"\pi"]\\
		\mathcal{D}_F\arrow[r,"\pi_*"]&\Delta'
	\end{tikzcd}
	\]
	The morphism $\pi_*:\mathcal{D}_F\rightarrow\Delta$ must be surjective since we have
	\[
	h^0(F,N_{F/\mathcal{X}_t})=d\quad and\quad h^0(F,N_{F/\mathcal{X}})=d+1
	\]
	This implies that $F$ deforms with the deformation $\pi:\pi^{-1}(\Delta')\longrightarrow\Delta'$.
	
	\bigskip
	
	\noindent Now we prove the claim.  Since $\mathcal{X}_t$ is the fibre of the deformation family $\pi:\mathcal{X}\rightarrow\Delta$, we have 
	\[
	N_{\mathcal{X}_t/\mathcal{X}}|_F\simeq \mathcal{O}_F
	\] 
	It is sufficient to compute the vector bundle $N_{F/\mathcal{X}_t}$. Consider the exact sequence
	\begin{equation}
		0\longrightarrow N_{F/\mathcal{X}_t}
		\longrightarrow N_{F/\mathcal{X}}\longrightarrow N_{\mathcal{X}_t/\mathcal{X}}|_F\longrightarrow 0
		\label{Exact sequence normal bundle}.
	\end{equation}
	We prove that the exact sequence splits and deduce the normal bundle $N_{F/\mathcal{X}}$ from $N_{F/\mathcal{X}_t}$ Finally, we compute the cohomology group with the known information of $F$, $N_{F/\mathcal{X}_t}$ and $N_{F/\mathcal{X}}$.
	
	\bigskip
	
	\noindent \textbf{The First Case: The  extremal ray $\mathbb{R}^+[\Gamma_{i_0}]$ is small.}
	
	\smallskip
	
	\noindent If the extremal ray $\mathbb{R}^+[\Gamma_{i_0}]$ is small, by \cite{zbMATH01353483} or \cite[Thm.1.1]{zbMATH04079590}, every connected component $F$  in the exceptional locus of $f:\mathcal{X}_t\longrightarrow Y$ is isomorphic to $\mathbb{P}^2$ with the normal bundle $N_{F/\mathcal{X}_t}\simeq\mathcal{O}_{\mathbb{P}^2}(-1)\oplus\mathcal{O}_{\mathbb{P}^2}(-1)$. 
	
	\noindent By \cite[Chap.\Romannum{3}.Rmk.7.14]{zbMATH03572315} and \cite[Chap.\Romannum{3}.Thm.5.1]{zbMATH03572315}, the short exact sequence (\ref{Exact sequence normal bundle}) splits and the subvariety $F$ satisfies \emph{Condition (A)}.
	
	\medskip
	
	\noindent \textbf{The second case: the extremal ray $\mathbb{R}^+[\Gamma_{i_0}]$ is divisorial.} 
	
	\noindent There are three types of contraction: $(3,0)$, $(3,1)$ and $(3,2)$.(cf. Definition \ref{Type introduction of generic contraction DEF})
	
	\smallskip
	
	\noindent \underline{The Type $(3,2)$:}
	
	\medskip
	
	\noindent By \cite[Thm.2.1]{zbMATH03882563}, the general fibre $F$ of $f|_{D_{\circ}}:D_{\circ}\longrightarrow B_{\circ}$ is isomorphic to $\mathbb{P}^1$. The  normal bundle is
	\[
	N_{F/\mathcal{X}_t}\simeq\mathcal{O}_{\mathbb{P}^1}(-1)\oplus\mathcal{O}_{\mathbb{P}^1}^{\oplus2}.
	\]
	By \cite[Chap.\Romannum{3}.Rmk.7.14]{zbMATH03572315} and \cite[Chap.\Romannum{3}.Thm.5.1]{zbMATH03572315}, the short exact sequence (\ref{Exact sequence normal bundle}) splits and we get
	\[
	N_{F/\mathcal{X}}\simeq\mathcal{O}_{\mathbb{P}^1}(-1)\oplus\mathcal{O}_{\mathbb{P}^1}^{\oplus3}
	\]
	Thus the general fibre $F$ satisfies \emph{Condition (A)}.
	
	\bigskip
	\noindent	\underline{The Type $(3,1)$:}
	
	\medskip
	
	\noindent  By \cite[Main Theorem]{zbMATH01224743}, the general fibre $F$ is either $\mathbb{P}^2$ or an irreducible quadric $Q$ in $\mathbb{P}^3$. Let $D$ be the prime divisor which is contracted by $f$. We have
	\[
	N_{D/\mathcal{X}_t}|_F\simeq \mathcal{O}_{\mathbb{P}^2}(-1)\quad or\quad N_{D/\mathcal{X}_t}|_F\simeq \mathcal{O}_{\mathbb{P}^2}(-2)
	\] 
	if $F$ is isomorphic to $\mathbb{P}^2$,
	and
	\[
	N_{D/\mathcal{X}_t}|_F\simeq\mathcal{O}_F(-1)
	\]
	if $F$ is isomorphic to an irreducible quadric $Q$ in $\mathbb{P}^3$. Consider the short exact sequence
	\[
	\begin{tikzcd}[column sep=small]
		0\arrow[r]&N_{F/D}\arrow[r]&N_{F/\mathcal{X}_t}\arrow[r]&N_{D/\mathcal{X}_t}|_F\arrow[r]&0
	\end{tikzcd}
	\]
	By Theorem \ref{Cohomology vanishing of degree 2 hypersurface THM}, \cite[Chap.\Romannum{3}.Rmk.7.14]{zbMATH03572315} and \cite[Chap.\Romannum{3}.Thm.5.1]{zbMATH03572315}, the short exact sequence splits and we get
	\[
	N_{F/\mathcal{X}_t}\simeq \mathcal{O}_{\mathbb{P}^2}(-1)\oplus\mathcal{O}_{\mathbb{P}^2}\quad or\quad N_{F/\mathcal{X}_t}\simeq \mathcal{O}_{\mathbb{P}^2}(-2)\oplus\mathcal{O}_{\mathbb{P}^2}.
	\]
	if $F$ is isomorphic to $\mathbb{P}^2$,
	and
	\[
	N_{F/\mathcal{X}_t}\simeq\mathcal{O}_F(-1)\oplus\mathcal{O}_F
	\]
	if $F$ is isomorphic to an irreducible quadric $Q$ in $\mathbb{P}^3$. Hence, the short exact sequence (\ref{Exact sequence normal bundle}) splits and we have
	\[
	N_{F/\mathcal{X}}\simeq \mathcal{O}_{\mathbb{P}^2}(-1)\oplus\mathcal{O}_{\mathbb{P}^2}^{\oplus 2}\quad or\quad N_{F/\mathcal{X}_t}\simeq \mathcal{O}_{\mathbb{P}^2}(-2)\oplus\mathcal{O}_{\mathbb{P}^2}^{\oplus 2}.
	\]
	if $F$ is isomorphic to $\mathbb{P}^2$,
	and
	\[
	N_{F/\mathcal{X}_t}\simeq\mathcal{O}_F(-1)\oplus\mathcal{O}^{\oplus2}_F
	\]
	if $F$ is isomorphic to an irreducible quadric $Q$ in $\mathbb{P}^3$.
	Again by Theorem \ref{Cohomology vanishing of degree 2 hypersurface THM}, \cite[Chap.\Romannum{3}.Rmk.7.14]{zbMATH03572315} and \cite[Chap.\Romannum{3}.Thm.5.1]{zbMATH03572315}, the general fibre $F$ satisfies \emph{Condition (A)}.
	
	\bigskip
	
	\noindent \underline{The Type $(3,0)$:}
	
	\medskip
	
	\noindent  Let $F$ be the exceptional divisor and we prove $F$ is the subvariety which satisfies \emph{Condition (A)}.
	By \cite[Thm.2.1]{zbMATH03882563}, the variety $F$ must be one of the following cases:
	\begin{itemize}
		\item[1.] $F$ is isomorphic to $\mathbb{P}^3$ and the normal bundle is one of the following cases:
		\[
		N_{F/\mathcal{X}_t}\simeq\mathcal{O}_{\mathbb{P}^3}(-1),\; \mathcal{O}_{\mathbb{P}^3}(-2)\; or \;\mathcal{O}_{\mathbb{P}^3}(-3)
		\]
		\item[2.] $F$ is isomorphic to an irreducible quadric in $\mathbb{P}^4$ and the normal bundle is either $\mathcal{O}_F(-1)$ or $\mathcal{O}_F(-2)$
		\item[3.] $F$ is a Del Pezzo variety. The line bundles $K_{\mathcal{X}_t}$ and $F|_F$ are numerically equivalent. 
	\end{itemize}
	When $F$ is $\mathbb{P}^3$ or an irreducible quadric in $\mathbb{P}^4$, by Theorem \ref{Cohomology vanishing of degree 2 hypersurface THM}, \cite[Chap.\Romannum{3}.Rmk.7.14]{zbMATH03572315} and \cite[Chap.\Romannum{3}.Thm.5.1]{zbMATH03572315}, the short exact sequence (\ref{Exact sequence normal bundle}) splits and we get
	\[
	N_{F/\mathcal{X}}\simeq N_{F/\mathcal{X}_t}\oplus \mathcal{O}_F.
	\]
	Therefore, the subvariety $F$ satisfies \emph{Condition (A)} except for the third case.
	Now, we treat with the third case: $F$ is a Del Pezzo variety of dimension three. The line bundles $K_{\mathcal{X}_t}|_F$ and $F|_F$ are numerically equivalent.
	
	\bigskip
	
	\noindent\textbf{Claim:} When $F$ is a Del Pezzo variety and the line bundles $K_{\mathcal{X}_t}|_F$ and $F|_F$ are numerically equivalent, it satisfies \emph{Condition (A)}.
	\begin{proof}
		It is sufficient to prove that $H^1(F,\mathcal{O}_F)=0$ and $H^1(F,\mathcal{O}_F(F))=0$. 
		
		\noindent We first prove that, for all $i>0$, we have
		\[
		H^i(F,\mathcal{O}_F)=0
		\]
		Consider the long exact sequence of the direct image of sheaves induced by the short exact sequence
		\[
		\begin{tikzcd}[sep=small]
			0\arrow[r]&\mathcal{O}_{\mathcal{X}_t}(-F)\arrow[r]&\mathcal{O}_{\mathcal{X}_t}\arrow[r]&\mathcal{O}_F\arrow[r]&0
		\end{tikzcd}
		\]
		By Theorem  \ref{Generic contraction exists for 4-fold THM}, the line bundles $-K_{\mathcal{X}_t}$ and $\mathcal{O}_{\mathcal{X}_t}(-F)$ are $f$-ample. Therefore, by relative Kawamata-Viehweg vanishing Theorem \cite[Thm.1.1,(\romannum{2})]{zbMATH08084092}, for $i>0$
		We know:
		\[
		R^if_*\mathcal{O}_{\mathcal{X}_t}=R^if_*\mathcal{O}_{\mathcal{X}_t}(-F)=0,
		\]
		Hence, for all $i>0$, one has $H^i(F,\mathcal{O}_F)=R^if_*\mathcal{O}_F=0$.
		
		\noindent Now we start to prove $H^1(F,\mathcal{O}_F(F))=0$. By the following exact sequence twisting with $\mathcal{O}_{\mathcal{X}_t}(F)$  
		\[
		\begin{tikzcd}[sep=small]
			0\arrow[r]&\mathcal{O}_{\mathcal{X}_t}(-F)\arrow[r]&\mathcal{O}_{\mathcal{X}_t}\arrow[r]&\mathcal{O}_F\arrow[r]&0,
		\end{tikzcd}
		\]
		we get
		\[
		\begin{tikzcd}[sep=small]
			0\arrow[r]&\mathcal{O}_{\mathcal{X}_t}\arrow[r]&\mathcal{O}_{\mathcal{X}_t}(F)\arrow[r]&\mathcal{O}_F(F)\arrow[r]&0.
		\end{tikzcd}
		\]
		Consider the long exact sequence of the direct image of sheaves induced by the short exact sequence
		\[
		\begin{tikzcd}[sep=small]
			0\arrow[r]&\mathcal{O}_{\mathcal{X}_t}\arrow[r]&\mathcal{O}_{\mathcal{X}_t}(F)\arrow[r]&\mathcal{O}_F(F)\arrow[r]&0.
		\end{tikzcd}
		\]
		Since,  $R^if_{*}\mathcal{O}_{\mathcal{X}_t}$ vanishes, for all $i>0$,  we have
		\[
		R^if_{*}\mathcal{O}_{\mathcal{X}_t}(F)\simeq R^if_{*}\mathcal{O}_F(F)\simeq H^1(F,\mathcal{O}_F(F))
		\]
		By assumption,  $(-K_{\mathcal{X}_t}+F)|_F$ are numerically trivial. Therefore, by the definition of the contraction of the extremal ray $[\Gamma_{i_0}]$, we have
		\[
		(-K_{\mathcal{X}_t}+F)\cdot\Gamma_{i_0}=0.
		\] 
		This implies that there exists a line bundle $M$ over $Y$ such that 
		\[
		f^*M\simeq -K_{\mathcal{X}_t}+F
		\]
		by Lemma \ref*{relatively numerical trivial line bundle is the pull back of the line bundle on the base LEM}.
		Hence, we have
		\[
		R^1f_{*}\mathcal{O}_{\mathcal{X}_t}(K_{\mathcal{X}_t}-K_{\mathcal{X}_t}+F)=R^1f_{*}(K_{\mathcal{X}_t}+f^*M)=R^1f_{*}K_{\mathcal{X}_t}\otimes M
		\]
		Yet the sheaf $R^1f_{*}K_{\mathcal{X}_t}$ vanishes by Grauert-Riemenschneider vanishing Theorem\cite[Thm.2.2.2]{zbMATH00703571}, so $R^1f_*\mathcal{O}_{\mathcal{X}_t}(F)=0.$
		Therefore,  we have $H^1(F,\mathcal{O}_F(F))=0$
	\end{proof}  
	\noindent The ampleness of $-K_{\mathcal{X}_t}|_F$ is a consequence of that $-K_{\mathcal{X}_t}$ is $f$-ample.
\end{proof}

\begin{remark}
{\rm
	We cite \cite[Thm.2.1]{zbMATH03882563} for the classification of the fibres of the contraction in type $(3,0)$ and $(3,2)$. In the paper, he did use the property that the variety itself is projective. However, during our proof of Theorem \ref{Generic contraction exists for 4-fold THM}, we show that the prime divisor $D$ fulfills all the condition he obtained from the contraction theorem of a projective variety, and thus the classification remains valid.
}
\end{remark}

\begin{proof}[Proof of Theorem \ref{nefness preserves by deformation THM}]
	We argue by contradiction. Let $\Delta'$ be the small neighborhood that we selected in Proposition  \ref{countable many fibres with canonical divisor are not nef PRO}.
	We assume that there exists a point $t\in\Delta'$ such that $K_{\mathcal{X}_t}$ is not nef. By Proposition \ref{Pseudo-effective preserves in the deformation PRO}, we know $K_{\mathcal{X}_t}$ is pseudo-effective but not nef. 
	
	\noindent By Theorem \ref{general fibre of contraction deforms with the deformation family THM}, there is an irreducible subvariety $F\subset\mathcal{X}_t$ which deforms with $\mathcal{X}_t$. More precisely, the deformation family of $F$ is parameterized by an irreducible subvariety $\mathcal{D}$ in the relatively Chow variety $\text{Chow}(\pi^{-1}(\Delta')/\Delta')$
	such that the morphism $\pi_*:\mathcal{D}\longrightarrow\Delta'$ is surjective. Let $\mathcal{U}$ be the universal family of $\mathcal D$. Consider the diagram
	\[
	\begin{tikzcd}
		\mathcal{U}\arrow[r,"pr_2"]\arrow[d,"\mu",swap]&\pi^{-1}(\Delta')\arrow[d,"\pi"]\\
		\mathcal{D}\arrow[r,"\pi_*"]&\Delta'
	\end{tikzcd}
	\]
	Without the loss of generality, we assume $0\in\mathcal{D}$ and $\mu^{-1}(0)$ is the subvariety $F$. By Theorem \ref{Generic contraction exists for 4-fold THM} and Theorem \ref{general fibre of contraction deforms with the deformation family THM},, we know $pr_2^*(-K_{\mathcal{X}})|_F$ is ample. Yet ampleness is an open property, the flatness of the morphism $\mu:\mathcal{U}\longrightarrow\mathcal{D}$ implies that the line bundle $pr_2^*(-K_{\mathcal{X}})|_{\mu^{-1}(s)}$ is ample for $s$ contained in a small neighborhood $B(0,\varepsilon)\subset\mathcal{D}$. 
	Now we choose $s\in B(0,\varepsilon)$ to be very general, since the morphism $\pi_*:\mathcal{D}\longrightarrow\Delta'$ is surjective, the variety $pr_2(\mu^{-1}(s))$ is in a very general fibre of $\pi:\pi^{-1}(\Delta')\longrightarrow\Delta'$. The line bundle $-K_{\mathcal{X}}|_{pr_2(\mu^{-1}(s))}$ is ample by the projection formula.
	
	Hence, there is a  small neighborhood $B\subset\Delta$ of $t$, for every point $b\in B$, there always a subvariety $F_b$ contained in $\mathcal{X}_b$ such that the line bundle
	\[
	-K_{\mathcal{X}}|_{F_b}\simeq -K_{\mathcal{X}_b}|_{F_b}
	\]
	is ample. 
	This implies that $K_{\mathcal{X}_b}$ is not nef, a contradiction to Proposition \ref{countable many fibres with canonical divisor are not nef PRO}.
\end{proof}
	\bibliographystyle{alpha}
	\bibliography{reference_deformation}

@misc{arXiv:2501.18088,
	author = {Ou, Wenhao},
	title = {A characterization of uniruled compact {K{\"a}hler} manifolds},
	year = {2025},
	howpublished = {Preprint, {arXiv}:2501.18088 [math.{AG}] (2025)},
	url = {https://arxiv.org/abs/2501.18088},
	arXiv = {arXiv:2501.18088}
}

@misc{arXiv:2404.12007,
	author = {Hacon, Christopher and Paun, Mihai},
	title = {On the {Canonical} {Bundle} {Formula} and {Adjunction} for {Generalized} {Kaehler} {Pairs}},
	year = {2024},
	howpublished = {Preprint, {arXiv}:2404.12007 [math.{AG}] (2024)},
	url = {https://arxiv.org/abs/2404.12007},
	arXiv = {arXiv:2404.12007}
}

@article{zbMATH06541951,
	author = {H{\"o}ring, Andreas and Peternell, Thomas},
	title = {Minimal models for {K{\"a}hler} threefolds},
	fjournal = {Inventiones Mathematicae},
	journal = {Invent. Math.},
	issn = {0020-9910},
	volume = {203},
	number = {1},
	pages = {217--264},
	year = {2016},
	language = {English},
	doi = {10.1007/s00222-015-0592-x},
	zbMATH = {6541951},
	Zbl = {1337.32031}
}

@book{zbMATH01634463,
	author = {Debarre, Olivier},
	title = {Higher-dimensional algebraic geometry},
	fseries = {Universitext},
	series = {Universitext},
	issn = {0172-5939},
	isbn = {0-387-95227-6},
	year = {2001},
	publisher = {New York, NY: Springer},
	language = {English},
	zbMATH = {1634463},
	Zbl = {0978.14001}
}

@book{zbMATH00703571,
	author = {Beltrametti, Mauro C. and Sommese, Andrew J.},
	title = {The adjunction theory of complex projective varieties},
	fseries = {De Gruyter Expositions in Mathematics},
	series = {De Gruyter Expo. Math.},
	issn = {0938-6572},
	volume = {16},
	isbn = {3-11-014355-0},
	year = {1995},
	publisher = {Berlin: de Gruyter},
	language = {English},
	zbMATH = {703571},
	Zbl = {0845.14003}
}

@book{zbMATH00611963,
	editor = {Grauert, H. and Peternell, Th. and Remmert, R. and Gamkrelidze, R. V.},
	title = {Several complex variables {VII}. {Sheaf}-theoretical methods in complex analysis},
	fseries = {Encyclopaedia of Mathematical Sciences},
	series = {Encycl. Math. Sci.},
	issn = {0938-0396},
	volume = {74},
	isbn = {3-540-56259-1},
	year = {1994},
	publisher = {Berlin: Springer-Verlag},
	language = {English},
	zbMATH = {611963},
	Zbl = {0793.00010}
}

@article{zbMATH08084092,
	author = {Fujino, Osamu},
	title = {Vanishing theorems for projective morphisms between complex analytic spaces},
	fjournal = {Mathematical Research Letters},
	journal = {Math. Res. Lett.},
	issn = {1073-2780},
	volume = {32},
	number = {3},
	pages = {739--769},
	year = {2025},
	language = {English},
	doi = {10.4310/MRL.250728235924},
	zbMATH = {8084092}
}

@article{zbMATH01224743,
	author = {Takagi, Hiromichi},
	title = {Classification of extremal contractions from smooth fourfolds of {{\((3,1)\)}}-type},
	fjournal = {Proceedings of the American Mathematical Society},
	journal = {Proc. Am. Math. Soc.},
	issn = {0002-9939},
	volume = {127},
	number = {2},
	pages = {315--321},
	year = {1999},
	language = {English},
	doi = {10.1090/S0002-9939-99-05114-X},
	zbMATH = {1224743},
	Zbl = {0905.14009}
}

@article{zbMATH07147342,
	author = {Cao, Junyan and H{\"o}ring, Andreas},
	title = {Rational curves on compact {K{\"a}hler} manifolds},
	fjournal = {Journal of Differential Geometry},
	journal = {J. Differ. Geom.},
	issn = {0022-040X},
	volume = {114},
	number = {1},
	pages = {1--39},
	year = {2020},
	language = {English},
	doi = {10.4310/jdg/1577502017},
	zbMATH = {7147342},
	Zbl = {1442.14055}
}

@book{zbMATH00833161,
	author = {Koll{\'a}r, J{\'a}nos},
	title = {Rational curves on algebraic varieties},
	fseries = {Ergebnisse der Mathematik und ihrer Grenzgebiete. 3. Folge},
	series = {Ergeb. Math. Grenzgeb., 3. Folge},
	issn = {0071-1136},
	volume = {32},
	isbn = {3-540-60168-6},
	year = {1995},
	publisher = {Berlin: Springer-Verlag},
	language = {English},
	zbMATH = {833161},
	Zbl = {0877.14012}
}

@article{zbMATH04079590,
	author = {Kawamata, Yujiro},
	title = {Small contractions of four dimensional algebraic manifolds},
	fjournal = {Mathematische Annalen},
	journal = {Math. Ann.},
	issn = {0025-5831},
	volume = {284},
	number = {4},
	pages = {595--600},
	year = {1989},
	language = {English},
	doi = {10.1007/BF01443353},
	url = {https://eudml.org/doc/164574},
	zbMATH = {4079590},
	Zbl = {0661.14009}
}

@article{zbMATH01353483,
	author = {Andreatta, Marco and Wi{\'s}niewski, Jaros{\l}aw A.},
	title = {On contractions of smooth varieties},
	fjournal = {Journal of Algebraic Geometry},
	journal = {J. Algebr. Geom.},
	issn = {1056-3911},
	volume = {7},
	number = {2},
	pages = {253--312},
	year = {1998},
	language = {English},
	zbMATH = {1353483},
	Zbl = {0966.14012}
}

@article{zbMATH06216389,
	author = {Greb, Daniel and Lehn, Christian and Rollenske, S{\"o}nke},
	title = {Lagrangian fibrations on hyperk{\"a}hler manifolds -- question of {Beauville}},
	fjournal = {Annales Scientifiques de l'{\'E}cole Normale Sup{\'e}rieure. Quatri{\`e}me S{\'e}rie},
	journal = {Ann. Sci. {\'E}c. Norm. Sup{\'e}r. (4)},
	issn = {0012-9593},
	volume = {46},
	number = {3},
	pages = {375--403},
	year = {2013},
	language = {English},
	url = {smf4.emath.fr/en/Publications/AnnalesENS/4_46/html/ens_ann-sc_46_375-403.php},
	zbMATH = {6216389},
	Zbl = {1281.32016}
}

@article{zbMATH03882563,
	author = {Ando, Tetsuya},
	title = {On extremal rays of the higher dimensional varieties},
	fjournal = {Inventiones Mathematicae},
	journal = {Invent. Math.},
	issn = {0020-9910},
	volume = {81},
	pages = {347--357},
	year = {1985},
	language = {English},
	doi = {10.1007/BF01389057},
	url = {https://eudml.org/doc/143259},
	zbMATH = {3882563},
	Zbl = {0554.14001}
}

@book{zbMATH03572315,
	author = {Hartshorne, Robin},
	title = {Algebraic geometry},
	fseries = {Graduate Texts in Mathematics},
	series = {Grad. Texts Math.},
	issn = {0072-5285},
	volume = {52},
	year = {1977},
	publisher = {Springer, Cham},
	language = {English},
	zbMATH = {3572315},
	Zbl = {0367.14001}
}

@article{zbMATH05636215,
	author = {Wi{\'s}niewski, Jaroslaw A.},
	title = {Rigidity of the {Mori} cone for {Fano} manifolds},
	fjournal = {Bulletin of the London Mathematical Society},
	journal = {Bull. Lond. Math. Soc.},
	issn = {0024-6093},
	volume = {41},
	number = {5},
	pages = {779--781},
	year = {2009},
	language = {English},
	doi = {10.1112/blms/bdp025},
	zbMATH = {5636215},
	Zbl = {1193.14021}
}

@misc{arXiv:2510.23967,
	author = {Mu-Lin Li and Sheng Rao and Kai Wang},
	title = {Deformation of nef adjoint canonical line bundles},
	year = {2025},
	howpublished = {Preprint, {arXiv}:2510.23967 [math.{AG}] (2025)},
	url = {https://arxiv.org/abs/2510.23967},
	arXiv = {arXiv:2510.23967}
}

@article{zbMATH00058549,
	author = {Wi{\'s}niewski, Jaros{\l}aw A.},
	title = {On deformation of nef values},
	fjournal = {Duke Mathematical Journal},
	journal = {Duke Math. J.},
	issn = {0012-7094},
	volume = {64},
	number = {2},
	pages = {325--332},
	year = {1991},
	language = {English},
	doi = {10.1215/S0012-7094-91-06415-X},
	zbMATH = {58549},
	Zbl = {0773.14003}
}

@article{zbMATH03935471,
	author = {Bingener, J{\"u}rgen},
	title = {On deformations of {K{\"a}hler} spaces. {II}},
	fjournal = {Archiv der Mathematik},
	journal = {Arch. Math.},
	issn = {0003-889X},
	volume = {41},
	pages = {517--530},
	year = {1983},
	language = {English},
	doi = {10.1007/BF01198581},
	url = {https://eudml.org/doc/173300},
	zbMATH = {3935471},
	Zbl = {0584.32043}
}

@book{zbMATH05016783,
	author = {Morrow, James and Kodaira, Kunihiko},
	title = {Complex manifolds},
	edition = {Reprint with corrections of the 1971 original},
	isbn = {0-8218-4055-X},
	year = {2006},
	publisher = {Providence, RI: AMS Chelsea Publishing},
	language = {English},
	zbMATH = {5016783},
	Zbl = {1087.32501}
}

@misc{li2026deformationinvariancecanonicalnefness,
	title={Deformation invariance of canonical nefness in smooth Kahler morphisms}, 
	author={Mu-Lin Li and Xiao-Lei Liu and Sheng Rao},
	year={2026},
	eprint={2609.14435},
	archivePrefix={arXiv},
	primaryClass={math.AG},
	url={https://arxiv.org/abs/2609.14435}, 
}

@misc{hacon2026kahlermmptranscendentalbasepointfree,
	title={On the K\"ahler MMP and the transcendental base-point-free theorem}, 
	author={Christopher Hacon and Lingyao Xie},
	year={2026},
	eprint={2607.24986},
	archivePrefix={arXiv},
	primaryClass={math.AG},
	url={https://arxiv.org/abs/2607.24986}, 
}

@article{zbMATH03935470,
	author = {Bingener, J{\"u}rgen},
	title = {On deformations of {K{\"a}hler} spaces. {I}},
	fjournal = {Mathematische Zeitschrift},
	journal = {Math. Z.},
	issn = {0025-5874},
	volume = {182},
	pages = {505--535},
	year = {1983},
	language = {English},
	doi = {10.1007/BF01215480},
	url = {https://eudml.org/doc/173300},
	zbMATH = {3935470},
	Zbl = {0584.32042}
}
\end{document}